\documentclass[preprint,11pt]{elsarticle}
\usepackage{amssymb}
\usepackage{url}
\usepackage{hyperref}
\hypersetup{colorlinks,linkcolor={blue},citecolor={blue},urlcolor={red}}  
\usepackage{amsmath}
\usepackage{algorithm}
\usepackage{color}
\usepackage{algpseudocode}
\usepackage{amsthm,setspace}
\newtheorem{theorem}{Theorem}
 \usepackage{lscape}
 \newtheorem{definition}{Definition}
\newtheorem{corollary}{Corollary}
\usepackage[margin=2.8cm]{geometry}
\setcitestyle{authoryear,open={(},close={)}}
\usepackage{lineno}
\makeatletter
\def\ps@pprintTitle{%
 \let\@oddhead\@empty
 \let\@evenhead\@empty
 \def\@oddfoot{\reset@font\hfil\thepage\hfil}
 \let\@evenfoot\@oddfoot
}
\makeatletter

\begin{document}
\begin{frontmatter}

\title{On testing the class of symmetry using entropy characterization and empirical likelihood approach}

\author[label1]{Ganesh Vishnu Avhad }
\author[label1]{ Ananya Lahiri}
\author[label2]{ Sudheesh K. Kattumannil}  \ead{skkattu@isichennai.res.in}
\address[label1]{Department of Mathematics and Statistics, Indian Institute of Technology, Tirupati, Andhra Pradesh, India}
\address[label2]{Statistical Sciences Division, Indian Statistical Institute, Chennai, Tamil Nadu, India}

 \doublespacing
\begin{abstract}
 In this paper, we obtain a new characterization result for symmetric distributions based on the entropy measure. Using the characterization, we propose a nonparametric test to test the symmetry of a distribution. We also develop the jackknife empirical likelihood and the adjusted jackknife empirical likelihood ratio tests. The asymptotic properties of the proposed test statistics are studied. We conduct extensive Monte Carlo simulation studies to assess the finite sample performance of the proposed tests. The simulation results indicate that the jackknife empirical likelihood and adjusted jackknife empirical likelihood ratio tests show better performance than the existing tests. Finally, two real data sets are analysed to illustrate the applicability of the proposed tests.
\end{abstract}

\end{frontmatter}
\doublespacing

\section{Introduction}\label{sec1}
\noindent In probability and statistics, the assessment of distributional symmetry is a fundamental problem. Many classical statistical procedures, such as the sign test and Wilcoxon signed-rank test, depend on the assumption of symmetry to ensure optimal performance. Methods such as trimmed means and $M$-estimators are constructed, particularly under the assumption of symmetry in the data.   Symmetry improves the effectiveness of resampling methods such as the bootstrap by accelerating the convergence of confidence intervals via the pivotal quantity. Symmetrically distributed pivotal quantities often lead to more accurate inferential outcomes, making the assumption of symmetry advantageous in a range of applications. In finance, models such as the Sharpe–Lintner capital asset pricing model and the Black–Scholes option pricing model heavily rely on symmetric about zero assumptions (\cite{davison2017symmetry}). The literature extensively examines various aspects of the symmetry of probability distributions, with numerous authors investigating the properties of symmetric distributions (see, e.g. \cite{ushakov2011one}, \cite{ahmadi2021characterization}, \cite{husseiny2024analyzing} and the references therein). Symmetric distributions are utilized by traders in the financial sector to forecast the value of assets over time, with the mean reversion hypothesis positing that asset prices will eventually revert to their long-term mean or average values based on the assumption of a symmetric distribution of prices over time, see \cite{ahmed2018stock}. 

The broad applicability of symmetric assumptions in statistical modelling has been well-documented; however, \cite{partlett2017measuring} reported potential issues with the indiscriminate fitting of symmetric distributions to data, emphasizing the potential risks associated with such practices. Therefore, the first step in any robust statistical inference involving symmetric distributions is to evaluate whether the data align sufficiently with the symmetry assumption. As a result of its wide range of practical uses, several goodness-of-fit tests for symmetric distributions using different approaches are available in the literature. The most famous classical sign and Wilcoxon signed-rank test (see, e.g. \cite{feuerverger1977empirical} and \cite{maesono1987competitors}), and its modified version can be found in \cite{vexler2023implement}. Symmetry holds significant relevance in non-parametric statistics and structured models, where a common assumption is that errors are distributed symmetrically. In the regression setup, \cite{allison2017monte} conducted an extensive Monte Carlo simulation study to test whether the error distribution in the linear regression model is symmetric. Additionally, several papers offer comparative analyses of symmetry tests; see  \cite{maesono1987competitors},   \cite{Patrick2006}, \cite{milovsevic2019comparison} and \cite{ivanovic2020comparison}. Characterization theorems play a crucial role, particularly in the development of goodness-of-fit tests; one may refer to \cite{marchetti2002characterization} and \cite{nikitin2017tests}.
Tests based on characterization provide a robust framework for assessing symmetry and have been applied to various families of distributions. Recently, \cite{milovsevic2016characterization}, \cite{ahmadi2019characterization} and \cite{bovzin2020new} proposed characterization-based tests for the univariate symmetry of the distribution. 
The entropy measures serve as a compelling alternative to traditional symmetry measures (\cite{fashandi2012characterizations}). Entropy, a notion built upon information theory, quantifies the uncertainty associated with a probability distribution. It captures the inherent randomness and disorder within the data, providing an in-depth perspective on distributional characteristics. Several forms of entropy, such as Shannon entropy, Rényi entropy, and Tsallis entropy, have been developed to address different aspects of uncertainty. Recently, several papers have been published on the characterization of distribution functions using the entropies of order statistics and record values. \cite{ahmadi2019characterization} introduced several characterization results for symmetric continuous distributions, which are established based on the properties of order statistics and various information measures. \cite{ahmadi2020characterization} obtained a characterization of the symmetric continuous distributions based on the properties of $k$-records and spacings. \cite{ahmadi2021characterization}, \cite{jose2022symmetry} and \cite{gupta2024some} developed new tests for the distributional symmetry based on the different information measures. 
   
\cite{thomas1975confidence} introduced the first paper to find the empirical likelihood-based confidence intervals for survival data analysis. The empirical likelihood method is a powerful non-parametric approach that does not depend on specific parametric assumptions, and inference based on it does not require variance estimation. \cite{jing2009jackknife} introduced the jackknife empirical likelihood (JEL) approach, integrating two non-parametric techniques: the jackknife and empirical likelihood. Since then, various tests based on JEL have been developed and explored in the literature. Some of the recent works in this area can be found in \cite{liu2023general}, \cite{huang2024jackknife}, \cite{chen2024jackknife} and \cite{suresh2025jackknife}. For a comprehensive overview of the JEL method, including recent advancements, the reader is referred to the review articles by \cite{lazar2021review} and \cite{liu2023review}. In addition, the JEL method can face challenges due to the empty set problem, as noted by \cite{chen2008adjusted} and \cite{jing2017transforming}. To enhance the accuracy of JEL, several techniques have been introduced, including bootstrap calibration (\cite{owen1988empirical}) and Bartlett correction (\cite{chen2007second}). To address both the empty set issue and low coverage probability, the adjusted jackknife empirical likelihood (AJEL) method was proposed by \cite{chen2008adjusted} and \cite{wang2015resampling}. In what follows, we introduce a novel entropy-based characterization for the symmetric distributions and propose a class of goodness-of-fit tests for the symmetric distributions. 

The rest of the article is organized as follows. In Section $\ref{sec2}$, we provide characterization results of the symmetric distributions based on the entropy measure, and with a few examples, we demonstrate the results. In Section $\ref{sec3}$, based on the random samples $X_1, X_2,..., X_n$ from an unknown distribution $F$, nonparametric goodness-of-fit test procedure is developed to test $H_0: F \in F_{S} $ against $H_1: F \notin F_{S}$, where $F_{S}$ = \{$F$: $F$ is a class of symmetric distributions\}. We also propose JEL and AJEL ratio tests for distribution symmetry based on the characterization introduced in Section $\ref{sec2}$. In Section $\ref{sec4}$, the finite sample performance of the proposed tests is evaluated and compared with other symmetry tests through a Monte Carlo simulation study. Section $\ref{sec5}$ presents and analyses two real examples. Finally, Section $\ref{sec6}$ concludes with some closing remarks.

\section{Characterization result based on entropy measure} \label{sec2}
\noindent Since the test is based on the entropy measure, we begin by modifying the generalized entropy measure obtained by \cite{kattumannil2022generalized}. 
 
\begin{definition}\label{def1}
Consider an absolutely continuous random variable with real values $X$ with a distribution function $F$ with a finite mean. Furthermore, let $\phi(\cdot)$ be a function of $X$, and let $w(\cdot)$ be a positive weight function. The generalized cumulative residual entropy $(\mathcal{GCRE})$ of $X$ is defined as 
\begin{eqnarray}\label{eq1}
    \mathcal{GCRE}(X) &=& \int_{-\infty}^{\infty} w(u) E\big [ \phi(X)-\phi(u)|X \geq u \big] dF(u),
\end{eqnarray} 
and generalized cumulative entropy $(\mathcal{GCE})$ of \( X \) is defined as
\begin{eqnarray}\label{eq2}
    \mathcal{GCE}(X) &=& \int_{-\infty}^{\infty} w(u) E\big [ \phi(u)-\phi(X)|X\leq u \big] dF(u),
\end{eqnarray}
here, $w(\cdot) $ and $ \phi(\cdot)$ can be arbitrarily chosen under the existence of the above integral such that $\mathcal{GCRE}(X)$ and $\mathcal{GCE}(X)$ become concave. 
\end{definition}

We now present the characterization theorem using the above definition, and we will use it to make our tests. The following theorem gives a relationship of symmetric distribution with $\mathcal{GCRE}(X)$ and $\mathcal{GCE}(X)$.
 
\begin{theorem}\label{theorem1} The following two statements are equivalent under the conditions $w(\cdot)$ is a positive weight function, for $x<u$,  $\phi(\cdot)\geq \phi(u)$ and $x>u$, $\phi(\cdot)\leq \phi(u)$: \\
\( (i) \) $X$ is a symmetric random variable;\\
\((ii)\)  $\mathcal{GCRE}(X)=\mathcal{GCE}(X)$.
\end{theorem}

\begin{proof}
Firstly, to show \((i) \implies (ii)\), let \(X\) be a symmetric random variable. This implies that the distribution function \(F(x)\) is symmetric about a point \(u \in \mathbb{R}\), i.e.,  
\begin{equation}\label{sym}
    F(u + x) = 1 - F(u - x), \quad \text{and} \quad \Bar{F}(u + x) = F(u - x), \quad \forall x \in \mathbb{R}.
\end{equation}
Since \(X\) is symmetric about \(u\), it follows that
\begin{equation}\label{sym1}
E[\phi(X) - \phi(u) \mid X \geq u] = E[\phi(u) - \phi(X) \mid X \leq u].
\end{equation}
To show (\ref{sym1}), we have
\begin{align*}
E[\phi(X) - \phi(u) \mid X \geq u] &= \int_{u}^\infty \frac{\phi(x) - \phi(u)}{\Bar{F}(u)} \,dF(x).
\end{align*}
Applying the change of variable \(x = y + u\), and the identity \eqref{sym} evaluated at \(x = 0\) gives $F(u) = 1 - F(u)$, which leads to $F(u) = \frac{1}{2}$, and hence $F(u) = \Bar{F}(u)$. One has
\begin{align*}
E[\phi(X) - \phi(u) \mid X \geq u] &= \int_0^\infty \frac{\phi(u + y) - \phi(u)}{\Bar{F}(u)} \, dF(u + y).
\end{align*}
Next, applying the substitution \(y = u - x\) (so \(x = 2u - y\)), we have
\begin{align*}
E[\phi(X) - \phi(u) \mid X \geq u] &= \int_u^\infty \frac{\phi(2u - x) - \phi(u)}{F(u)} \, dF(2u - x).
\end{align*}
Letting \(z = 2u - x\) $\implies$ \(x = 2u - z\), we obtain
\begin{align*}
E[\phi(X) - \phi(u) \mid X \geq u] &= -\int_u^\infty \frac{\phi(z) - \phi(u)}{F(u)} \, dF(z) \\
&= \int_{-\infty}^u \frac{\phi(u) - \phi(z)}{F(u)} \, dF(z) \\
&= E[\phi(u) - \phi(X) \mid X \leq u].
\end{align*}
Therefore, it follows that
\begin{align*}
\mathcal{GCRE}(X) &= \int_{-\infty}^{\infty} w(u) \, E[\phi(X) - \phi(u) \mid X \geq u] \, dF(u) \\
&= \int_{-\infty}^{\infty} w(u) \, E[\phi(u) - \phi(X) \mid X \leq u] \, dF(u) \\
&= \mathcal{GCE}(X).
\end{align*}
Hence, \(\mathcal{GCRE}(X) = \mathcal{GCE}(X)\) when \(X\) is a symmetric random variable. Hence $(i) \implies (ii)$.

Next, to show $(ii) \implies (i)$, we have
\begin{align}\label{xx}
0&= \mathcal{GCRE}(X) - \mathcal{GCE}(X)\nonumber\\
&=\int_{-\infty}^{\infty} \int_{u}^{\infty} \frac{w(u)}{\Bar{F} (u)}  \Big(\phi(x)-\phi(u)\Big) dF(x)dF(u)-\int_{-\infty}^{\infty} \int_{-\infty}^{u}\frac{w(u)}{F(u)}  \Big(\phi(u)-\phi(x)\Big) dF(x)dF(u)\nonumber\\
&=  \int_{-\infty}^{\infty}  \frac{w(u)}{\Bar{F}(u)} \int_{u}^{\infty} \Big(\phi(x)-\phi(u)\Big) dF(x)dF(u)  -  \int_{-\infty}^{\infty} \frac{w(u)}{F(u)} \int_{-\infty}^{u} \Big(\phi(u)-\phi(x)\Big) dF(x)dF(u)\nonumber\\
&  = \int_{-\infty}^{\infty} \bigg(  \frac{1}{\Bar{F}(u)} \int_{u}^{\infty} \Big(\phi(x)-\phi(u)\Big) dF(x)   -   \frac{1}{F(u)} \int_{-\infty}^{u} \Big(\phi(u)-\phi(x)\Big) dF(x) \bigg) w(u)dF(u) \\
&= \int_{-\infty}^{\infty}g(u) w(u)dF(u) ~~ ~~(say).\nonumber
\end{align}
We know that if $g(u) \geq 0$ for all $u$ in some domain $D$, and 
$\int_D g(u)\, du = 0$, then $g(u) = 0$ almost everywhere (a.e.) on $D$.
Also, if $(\ref{xx})$ holds and it is noted that, under the conditions specified in the theorem, the function \(g(u)\) satisfies \(g(u) \geq 0\) for all $u$, then from the completeness property for the distribution of $U$, it follows that

\begin{eqnarray}\label{ness}
     &&\frac{1}{\Bar{F}(u)} \int_{u}^{\infty} \big(\phi(x)-\phi(u)\big) dF(x) = \frac{1}{F(u)} \int_{-\infty}^{u} \big(\phi(u)-\phi(x)\big) dF(x)\nonumber\\
     &\implies&\frac{1}{\Bar{F}(u)} \int_{u}^{\infty}  \phi(x) dF(x) -\phi(u) = \phi(u)-\frac{1}{F(u)} \int_{-\infty}^{u}  \phi(x)  dF(x)\nonumber\\
     &\implies&\frac{1}{\bar{F}(u)} \int_{u}^{\infty} \phi(x) dF(x) + \frac{1}{F(u)} \int_{-\infty}^{u} \phi(x) dF(x) = 2\phi(u).
\end{eqnarray}
Now, recall that the conditional expectation of \( \phi(X) \) given \( X \geq u \) is
\[
E[\phi(X) \mid X \geq u] = \frac{1}{\bar{F}(u)} \int_{u}^{\infty} \phi(x) dF(x),
\]
also, the conditional expectation of \( \phi(X) \) given \( X \leq u \) is
\[
E[\phi(X) \mid X \leq u] = \frac{1}{F(u)} \int_{-\infty}^{u} \phi(x) dF(x).
\]
Using these definitions, (\ref{ness}) simplifies to
\begin{eqnarray}\label{xx1} 
&&E[\phi(X) \mid X \geq u] + E[\phi(X) \mid X \leq u] = 2\phi(u)\nonumber\\
&\implies& E[\phi(X) \mid X \geq u] - \phi(u) = \phi(u) - E[\phi(X) \mid X \leq u]. 
\end{eqnarray}
Now, we assume without loss of generality $\phi(u)=u$. Hence, the mean residual life and the mean past life at time \( u \) are $E[X - u \mid X \geq u] =\frac{1}{\bar{F}(u)} \int_u^\infty \bar{F}(x) dx$ and $E[u - X \mid X \leq u] = \frac{1}{F(u)} \int_{-\infty}^u F(x) dx,$ respectively. From (\ref{xx1}), we have 
\begin{equation}\label{eqlty}
    \frac{1}{1 - F(u)} \int_u^\infty (1 - F(x)) dx = \frac{1}{F(u)} \int_{-\infty}^u F(x) dx.
\end{equation}
Now consider the transformation $x=u+y$ in the first integral, so that
  \begin{equation}\label{fstint}
       \int_u^\infty (1 - F(x)) dx = \int_0^\infty (1 - F(u + y)) dy,
  \end{equation}
and $x=u-y$ in the second integral, we get
\begin{equation}\label{secint}
  \int_{-\infty}^u F(x) dx = \int_0^\infty F(u - y) dy. 
\end{equation}
From (\ref{fstint}) and (\ref{secint}), we get
\begin{equation*}
\frac{1}{1 - F(u)} \int_0^\infty (1 - F(u + y)) dy = \frac{1}{F(u)} \int_0^\infty F(u - y) dy.
\end{equation*}
For simplicity, we can multiply both sides by \( F(u)(1 - F(u)) \) to get 
\begin{align*}
   &F(u) \int_0^\infty (1 - F(u + y)) dy  = (1 - F(u)) \int_0^\infty F(u - y) dy.\\
\implies\quad &F(u) \int_0^\infty (1 - F(u + y)) dy - (1 - F(u)) \int_0^\infty F(u - y) dy = 0. \\
\implies\quad & \int_0^\infty \Big(F(u)(1 - F(u + y)) - (1 - F(u)) F(u - y)\Big) dy = 0.\\
\implies\quad & \int_0^\infty G(y) dy = 0 ~\quad(say).
\end{align*}
If \( G(y) \geq 0 \) (or \( G(y) \leq 0 \)) for all \( y \geq 0 \), and the integral is zero, then it follows that \( G(y) = 0 \) for all \( y \geq 0 \). 
 Hence, we get
\begin{align}\label{lb}
  F(u)(1 - F(u + y)) &= (1 - F(u)) F(u - y), \quad u>0, ~y>0,   
\end{align}
for almost all $F(u) \in (0,1)$.
Following the technique of Theorem $2$ of \cite{ahmadi2020characterization}, we choose \( u = F^{-1}(1/2) \) (i.e., the point where \( F(u) = 1/2 \)), then (\ref{lb}) becomes
\begin{equation*} 
1 - F(u + y) = F(u - y), \quad y>0.  
\end{equation*}
This completes the proof $(ii)\implies (i)$. 
\end{proof}

We can define several entropy measures using (\ref{eq1}) and (\ref{eq2}), with different choices of $w(\cdot) $ and $ \phi(\cdot)$. For a detailed study, one can see \cite{kattumannil2022generalized} and the references therein. To construct the test for symmetry, we consider the weight function $w(u)=\Bar{F}(u)F(u)$ and $\phi(u)=u$.  
Then, by Theorem \ref{theorem1}, when $X$ has symmetric distributions if and only if  
\begin{equation}\label{char}
    \int_{-\infty}^{\infty} \Bar{F}(u)F(u)E\big[X-u | X \geq u\big]dF(u)= \int_{-\infty}^{\infty} \Bar{F}(u)F(u)E\big[u-X | X \leq u\big]dF(u).
\end{equation}

\subsection{Examples}
\noindent To illustrate the results obtained above, we have now explored a few examples.

\subsubsection*{1. Uniform distribution}
\noindent Let $X$  be a random variable following a uniform distribution on the interval $[-1,1]$ with the probability density function (pdf) and cumulative distribution function (cdf) given by 
  \begin{eqnarray*}
      f(x)= \dfrac{1}{2}, ~~
      \text{and}~~
      F(x)= \dfrac{x+1}{2}, ~~ -1\leq x \leq 1.
  \end{eqnarray*}
Hence
 \begin{eqnarray*}
      \int_{-1}^{1} \Bar{F}(u) F(u) E\big [ X-u|X \geq u \big] dF(u) &= &\int_{-1}^{1} F(u) \int_{u}^{1} (x-u) dF(x) dF(u) \\
                     &= &\int_{-1}^{1} \dfrac{(u+1)}{2} \int_{u}^{1} \dfrac{(x-u)}{4} dx du 
                     = \dfrac{1}{12},
\end{eqnarray*}
and
 \begin{eqnarray*}
\int_{-1}^{1} \Bar{F}(u)F(u) E\big[ u-X|X\leq u\big]dF(u)
                     &= &\int_{-1}^{1} \Bar{F}(u) \int_{-1}^{u} (u-x) dF(x) dF(u) \\
                     &= &\int_{-1}^{1} \dfrac{(1-u)}{2} \int_{-1}^{u} \dfrac{(u-x)}{4} dx du 
                     = \dfrac{1}{12}.\\
  \end{eqnarray*}

  \subsubsection*{2. Logistic distribution}
 \noindent The pdf and cdf of the logistic distribution is 
  \begin{eqnarray*}
      f(x)= \dfrac{e^{-x}}{(1+e^{-x})^2}, ~~
      \text{and}~~
      F(x)= \dfrac{1}{(1+e^{-x})}, ~~ -\infty \leq x \leq \infty.
  \end{eqnarray*}
One has
 \begin{eqnarray*}
\int_{-\infty}^{\infty} \Bar{F}(u) F(u) E\big [ X-u|X \geq u \big] dF(u)
      &=&\int_{-\infty}^{\infty} F(u)\int_{u}^{\infty}(x-u)dF(x) dF(u) \\
                     &= &\int_{-\infty}^{\infty} \int_{u}^{\infty} \dfrac{ (x-u)e^{-(x+u)}}{(1+e^{-x})^3(1+e^{-u})^2}dx du 
                     = \dfrac{3}{4}, 
 \end{eqnarray*}
 and
  \begin{eqnarray*}
 \int_{-\infty}^{\infty} \Bar{F}(u)F(u) E\big[ u-X|X\leq u\big]dF(u) 
    &=&\int_{-\infty}^{\infty} \Bar{F}(u) \int_{-\infty}^{u} (u-x) dF(x) dF(u) \\
                    &= &\int_{-\infty}^{\infty} \int_{-\infty}^{u}  \dfrac{(u-x)e^{-(2x+u)}}{(1+e^{-x})^3(1+e^{-u})^2}dx du 
                     = \dfrac{3}{4}.
  \end{eqnarray*}

\subsubsection*{3. Exponential distribution }
\noindent Let $X$ follow an exponential distribution with parameter $\lambda= 1$. Then 
\begin{eqnarray*}
     f(x)= e^{-x}, ~~
      \text{and}~~
      F(x)= 1-e^{-x}, ~~ x>0 .
\end{eqnarray*}
Hence
\begin{eqnarray*}
    \int_{0}^{\infty} \Bar{F}(u) F(u) E\big [ X-u|X \geq u \big] dF(u) 
      &=&\int_{0}^{\infty} F(u)\int_{u}^{\infty}(x-u)dF(x) dF(u) \\
                     &= &\int_{0}^{\infty} \int_{u}^{\infty} (x-u) (1-e^{-x}) e^{-(x+u)} dx du 
                     = \dfrac{5}{12}, 
 \end{eqnarray*} 
 and
 \begin{eqnarray*}
  \int_{0}^{\infty} \Bar{F}(u)F(u) E\big[ u-X|X\leq u\big]dF(u) 
    &=&\int_{0}^{\infty} \Bar{F}(u) \int_{0}^{u} (u-x) dF(x) dF(u) \\
                    &= &\int_{0}^{\infty} \Bar{F}(u) \int_{0}^{u} (u-x)e^{-(2x+u)} dF(x) dF(u) 
                     = \dfrac{1}{3}. 
  \end{eqnarray*}
We note that the integrals in (\ref{char}) are identical for the uniform and logistic distributions.
However, in the case of the exponential distribution, it is not equal, reflecting the inherent asymmetry of the distribution.

\section{Test for symmetry}\label{sec3}
\noindent In this section, we propose a non-parametric test to assess whether the underlying distribution of $X$ is symmetric. Let $\mathcal{F_{S}}$ = \{$F$: $F$ is a class of symmetric distributions\}. Based on random samples $X_1, X_2, \ldots, X_n$ from an unknown distribution function $F$, we are interested in testing the null hypothesis
$$H_0: F \in \mathcal{F_{S}} ~~ \text{against the alternative}~~ H_1: F \notin \mathcal{F_{S}}.$$  
Using (\ref{char}), we now introduce the departure measure: 
\begin{align}\label{eq.1}
    \Delta
    & =\int_{-\infty}^{\infty} \overline{F}(u)F(u)E[X-u | X \geq u]dF(u) - \int_{-\infty}^{\infty} \overline{F}(u)F(u)E[u-X | X \leq u]dF(u) \nonumber\\
   &= \Delta_1 - \Delta_2 \quad (say).  
\end{align}

In view of Theorem \ref{theorem1}, the departure measure $\Delta$ is zero under the null hypothesis and non-zero under the alternative hypothesis.  
For finding a test statistic, we simplify the departure measure $\Delta$ as   
 \begin{align}\label{delta1}
    \Delta_1 &= \int_{-\infty}^{\infty} F(u)\Bar{F}(u) E\big[X- u|X \geq u\big]dF(u) \nonumber\\
    &= \int_{-\infty}^{\infty}\int_{u}^{\infty} (x-u) F(u)dF(x)dF(u) \nonumber\\
    &= \int_{-\infty}^{\infty}\int_{u}^{\infty}x F(u)dF(x)dF(u)- \int_{-\infty}^{\infty}\int_{u}^{\infty} u F(u)dF(x)dF(u) \nonumber\\ 
    &= \int_{-\infty}^{\infty}\int_{-\infty}^{x}x F(u)dF(u)dF(x)- \int_{-\infty}^{\infty} u  F(u)\Bar{F}(u)dF(u) \nonumber\\ 
    &= \int_{-\infty}^{\infty}  \dfrac{xF^2(x)}{2}dF(x)- \int_{-\infty}^{\infty} u  F(u)(1-F(u))dF(u) \nonumber\\
    &= \dfrac{3}{2}\int_{-\infty}^{\infty}xF^2(x)dF(x)- \int_{-\infty}^{\infty} u  F(u)dF(u) \nonumber\\
     &= \dfrac{1}{2}E\big[\max(X_1, X_2, X_3)- \max(X_1, X_2)\big],
\end{align}   
and
\begin{align}\label{delta2}
    \Delta_2 &=\int_{-\infty}^{\infty} F(u)\Bar{F}(u) E\big[u-X|X\leq u\big]dF(u) \nonumber\\
     &= \int_{-\infty}^{\infty}\int_{-\infty}^{u} (u-x) \Bar{F}(u)dF(x)dF(u) \nonumber\\
    &= \int_{-\infty}^{\infty}\int_{-\infty}^{u} u \Bar{F}(u)dF(x)dF(u)- \int_{-\infty}^{\infty}\int_{-\infty}^{u} x \Bar{F}(u)dF(x)dF(u) \nonumber\\ 
    &= \int_{-\infty}^{\infty} u \Bar{F}(u) {F}(u) dF(u)- \int_{-\infty}^{\infty}\int_{x}^{\infty} x \Bar{F}(u)dF(u)dF(x) \nonumber\\ 
    &= \int_{-\infty}^{\infty} u  \Bar{F}(u)(1-\Bar{F}(u))dF(u)- \int_{-\infty}^{\infty}  \dfrac{x\Bar{F}^2(x)}{2}dF(x) \nonumber\\
    &= \int_{-\infty}^{\infty} u  \Bar{F}(u)dF(u)-\dfrac{3}{2}\int_{-\infty}^{\infty}x \Bar{F}^2(x)dF(x) \nonumber\\
     &= \dfrac{1}{2}E\big[\min(X_1, X_2) - \min(X_1, X_2, X_3)\big].
\end{align}    
Substituting $(\ref{delta1})$ and $(\ref{delta2})$ in (\ref{eq.1}), we obtain 
\begin{align}\label{delta}
    \Delta &= \dfrac{1}{2}E\big[\max(X_1, X_2, X_3)+\min(X_1, X_2, X_3)- \max(X_1, X_2)-\min(X_1, X_2)\big]. 
\end{align}
Now, one can simplify the above equation as  
\begin{eqnarray}\label{max}
    E(\max(X_1, X_2)+\min(X_1, X_2) )&=& \int_{-\infty}^{\infty} 2xF(x) dF(x) + \int_{-\infty}^{\infty} 2x\Bar{F}(x) dF(x)\nonumber\\
    &=& \int_{-\infty}^{\infty} 2x(F(x)+\Bar{F}(x))dF(x)\nonumber\\
    &=& \int_{-\infty}^{\infty} 2xdF(x) = 2E(X_1).
\end{eqnarray}
Hence (\ref{delta}) becomes 
\begin{align*}\label{deltafinal}
    \Delta &= \dfrac{1}{2}E\big[\max(X_1, X_2, X_3)+\min(X_1, X_2, X_3)-2 X_1\big]. 
\end{align*}
Consider a kernel 
\begin{eqnarray*}
    h^*(X_{1},X_{2}, X_{3}) &=& \dfrac{1}{2} \big( \max(X_1, X_2, X_3)+\min(X_1, X_2, X_3)\big)-X_1,
\end{eqnarray*}
such that $E(h^*(X_{1},X_{2}, X_{3}))= \Delta$.

Hence, the $U$-statistics based test for testing the symmetry is given by
     \begin{equation} \label{test statistic}
        \widehat{\Delta}_n = \frac{1}{\binom{n}{3}}\sum_{i=1 }^{n} \sum_{j=1; j < i }^{n}\sum_{k=1; k < j }^{n}  h(X_{i},X_{j}, X_{k}) 
     \end{equation}
     with 
     \begin{align}\label{h}
        h(X_{1},X_{2}, X_{3})& = \dfrac{1}{3}\bigg(\dfrac{3\max(X_1, X_2, X_3)+3\min(X_1, X_2, X_3)}{2}- (X_1+X_2+X_3)\bigg)
     \end{align}
     is the associated symmetric kernel of $h^*(X_{1},X_{2}, X_{3})$. Next, we express the proposed test statistic $\widehat{\Delta}_n$ in a simple form. Let $X_{(i)}$, $i=1,2,\ldots,n$ be the $i$-th order statistics based on a random sample of size $n$ from $F$. Using $X_{(i)}$, we have the following expressions 
\begin{eqnarray*}
    \sum_{i=1 }^{n} \sum_{j=1; j < i }^{n}\sum_{k=1; k < j }^{n}\max(X_i, X_j, X_k)&=& \dfrac{1}{2}\sum_{i=1}^{n} (i-1)(i-2)X_{(i)}, 
\end{eqnarray*}
and
\begin{eqnarray*}
    \sum_{i=1 }^{n} \sum_{j=1; j < i }^{n}\sum_{k=1; k < j }^{n}\min(X_i, X_j, X_k) &=& \dfrac{1}{2}\sum_{i=1}^{n} (n-i-1)(n-i)X_{(i)}.
\end{eqnarray*}

Hence, in terms of order statistics, the test statistic can be written as 
 \begin{small}
\begin{eqnarray*}
    \widehat{\Delta}_n&=& \dfrac{6}{n(n-1)(n-2)}\Bigg(\dfrac{3}{6\times 2}\sum_{i=1 }^{n} (i-1)(i-2)X_{(i)} +\dfrac{3}{6\times 2}\sum_{i=1 }^{n} (n-i-1)(n-i)X_{(i)}\Bigg) -\dfrac{1}{n}\sum_{i=1 }^{n} X_{(i)}\\
    &=& \dfrac{3}{2n(n-1)(n-2)}\Bigg(\sum_{i=1 }^{n} (i^2-3i+2)X_{(i)} +\sum_{i=1 }^{n} (n^2-2ni+i^2-n+i)X_{(i)} \Bigg) -\dfrac{1}{n}\sum_{i=1 }^{n} X_{(i)}\\
    &=&\dfrac{3}{2n(n-1)(n-2)}\Bigg(\sum_{i=1 }^{n} (n^2-2ni-n+2i^2-2i+2)X_{(i)}\Bigg)-\dfrac{1}{n}\sum_{i=1 }^{n} X_{(i)}\\
    &=&\dfrac{3}{2n(n-1)(n-2)}\Bigg(\sum_{i=1 }^{n} \Big(n(n-1)-2\big(i(n+1-i)-1\big)\Big)X_{(i)}\Bigg)-\dfrac{1}{n}\sum_{i=1 }^{n} X_{(i)}\\
    &=&\dfrac{3}{2n(n-2)}\sum_{i=1 }^{n} nX_{(i)}-\dfrac{2(n-2)}{2n(n-2)}\sum_{i=1 }^{n} X_{(i)}-\dfrac{3}{n(n-1)(n-2)}\sum_{i=1 }^{n} \big(i(n+1-i)-1\big)X_{(i)}\\
    &=&\dfrac{(3n-2n+4)}{2n(n-2)}\sum_{i=1 }^{n} X_{(i)}-\dfrac{3}{n(n-1)(n-2)}\sum_{i=1 }^{n} \big(i(n+1-i)-1\big)X_{(i)},
    \end{eqnarray*}
    \end{small}
which simplifies to
    \begin{eqnarray*}
    \widehat{\Delta}_n &=&\dfrac{(n+4)}{2n(n-2)}\sum_{i=1 }^{n} X_{(i)}-\dfrac{3}{n(n-1)(n-2)}\sum_{i=1 }^{n} \big(i(n+1-i)-1\big)X_{(i)}. 
    \end{eqnarray*}    

Next, by applying the asymptotic theory of $U$-statistics (refer to \cite{lee2019u}, Theorem $1$, page $76$), we derive the asymptotic distribution of $\widehat{\Delta}_n$.
\begin{theorem}\label{thm2}
    As $n \rightarrow \infty$,  $\sqrt{n}(\widehat{\Delta}_n-\Delta)$ converges in distribution to a normal random variable with mean zero and variance $\sigma^2= Var\big(K(X)\big)$, where 
\begin{align*}
       K(x) &=\bigg( xF(x)\big(F(x)-1\big)-\frac{x}{2}-  \int_{-\infty}^{x}y\big(1-2F(y)\big)dF(y)\bigg).  
     \end{align*}    
\end{theorem}
\begin{proof}
Since $\widehat{\Delta}_n$ is a $U$-statistic, according to the central limit theorem of $U$-statistics, we have the asymptotic normality of $\sqrt{n}(\widehat{\Delta}_n-\Delta)$ and the asymptotic variance is $9\sigma_1^2$, where $\sigma_1^2$ is given by
\begin{equation}\label{var}
    \sigma_1^2= Var\big(E(h(X_1,X_2,X_3)|X_1)\big),
\end{equation}
where the symmetric kernel function $h(\cdot)$ is given by (\ref{h}). Now, consider   
\begin{align*}
    E(h(X_1,X_2,X_3)|X_1) &= \dfrac{1}{2}E\Big( \max(x, X_2, X_3)+\min(x, X_2, X_3)\Big)-\dfrac{1}{3} E(x+X_2+X_3). 
\end{align*}
Let $Z=\max(X_2,X_3)$, and $W=\min(X_2, X_3)$ then
\begin{align*}
    E(\max(x, X_2, X_3)) &= E(xI(Z<x)+ZI(Z>x))\\&= xF^2(x)+2\int_{x}^{\infty}yF(y)dF(y),\\
    E(\min(x, X_2, X_3)) &= E(xI(W>x)+WI(W<x))\\&= x(1-F(x))^2+2\int_{-\infty}^{x}y\Bar{F}(y)dF(y).
\end{align*}    
  Hence 
\begin{align*}
    &E(\max(x, X_2, X_3)+\min(x, X_2, X_3))\\
    &= xF^2(x)+ x-2xF(x)+xF^2(x) + 2\int_{x}^{\infty}yF(y)dF(y) +2\int_{-\infty}^{x}y\Bar{F}(y)dF(y)\\
    &=2xF^2(x)+ x-2xF(x) +2\int_{-\infty}^{x}y{F}(y)dF(y) + 2\int_{x}^{\infty}yF(y)dF(y)  \\
    &\hspace{0.5cm}-2\int_{-\infty}^{x}y{F}(y)dF(y) +2\int_{-\infty}^{x}y(1-{F}(y))dF(y)\\
    &= 2xF^2(x)+ x-2xF(x)+ 2\int_{-\infty}^{\infty}y{F}(y)dF(y)- 4\int_{-\infty}^{x}y{F}(y)dF(y) +2\int_{-\infty}^{x}ydF(y), 
\end{align*}  
and 
\begin{equation*}
                 E(x+X_2+X_3) = x+2\int_{-\infty}^{\infty}ydF(y) . 
\end{equation*}
Substituting the above expressions in $(\ref{var})$, we obtain the variance expression given in the theorem.  
\end{proof}
Note that ${\Delta} = 0$ under the null hypothesis $H_0$.  Hence, we obtain the asymptotic null distribution of the test statistic. 
\begin{corollary}\label{cor1}
    Under $H_0$, as $n \rightarrow \infty$, $\sqrt{n}\widehat{\Delta}_n$ converges in distribution to a normal random variable with mean zero and variance $\sigma_{0}^2=Var(K_0(X))$, where $\sigma_{0}^2$ is the value of $\sigma^2$ evaluated under the null hypothesis.
\end{corollary}
Let $\widehat{\sigma}_{0}^2$ be a consistent estimator of the null variance $\sigma_0^2$. Using the asymptotic distribution, we can establish testing criteria based on the normal distribution. The null hypothesis $H_0$ is rejected in favor of the alternative hypothesis $H_1$ at a significance level $\alpha$, if
\begin{equation*}
    \frac{\sqrt{n}\lvert \widehat{\Delta}_n\lvert}{\widehat{\sigma}_0} > Z_{\alpha/2},
\end{equation*}
where, $Z_{\alpha}$ represents the upper $\alpha$-percentile point of a standard normal distribution. Implementing the normal-based test is challenging due to the difficulty in finding $\widehat{\sigma}_0$. Hence, we find the critical region of
the proposed test using a Monte Carlo simulation. We used the simulated critical region (SCR) technique to obtain critical points and perform the test to avoid relying on asymptotic critical values. The values of the lower quantile \( C_{1} \) and the upper quantile \( C_{2} \) are determined using the exact distribution in such a way that \( P(\widehat{\Delta}_n<C_{1})= P(\widehat{\Delta}_n>C_{2}) = \frac{\alpha}{2} \). The procedure to find the critical points in this study is outlined in Algorithm~1.

\begin{algorithm}  
\caption{Nonparametric Bootstrap Algorithm to Find $C_1$ and $C_2$}\label{alg:bootstrap}
\begin{algorithmic}
\State $x_i\gets \text{generate from null distribution}$
\State $n \gets \text{length}(x_i)$
\State $\widehat{\Delta}_n \gets \text{Calculate test statistic } \widehat{\Delta}_n(x)$
\State $B \gets 1000$ \Comment{Number of bootstrap replicates}
\State for $(b ~\text{from}~ 1~\text{to}~ B )$\{
\State $ i \gets$ sample($1:n$, size=$n$, replace=TRUE) 
\State  $ y \gets x[i]$
\State deltas[b] $\gets$ $\widehat{\Delta}_n$\}
\State $\text{deltas} \gets \text{sort}(\text{deltas})$
\State $C_1 \gets$ quantile(deltas, $\alpha/2$) \Comment{lower bound}
\State $C_2 \gets$ quantile(deltas, $1-\alpha/2$)  \Comment{upper bound}
\State ifelse(($\widehat{\Delta}_n$ $< C_1||$$\widehat{\Delta}_n$ $> C_2$), print(``Reject $H_0$"), print(``Accept $H_0$"))
\end{algorithmic}
\end{algorithm}

Consequently, we developed the distribution-free test known as the JEL ratio test to assess the distributional symmetry.  The JEL transforms the test statistic of interest into a sample mean using jackknife pseudo-values, proving particularly effective in addressing tests based on $U$-statistics (\cite{peng2018jackknife} and \cite{garg2024jackknife}). This approach is most suitable for dealing with the testing problem associated with a class of distributions where the null hypothesis is not completely specified. For recent developments related to this area, one can see \cite{kattumannil2024jackknife} and \cite{avhad2025}. Motivated by this, we formulate JEL and AJEL ratio tests in the following section.  

\subsection{ Jackknife empirical likelihood ratio test}
\noindent To construct the JEL ratio test, we initially define the jackknife pseudo values utilized in empirical likelihood. The jackknife pseudo values denoted by $\mathcal{V}_i,\hspace{0.2cm} i=1,\ldots, n$ are defined as
\begin{equation}\label{vi}
    \widehat{\mathcal{V}}_i= n\widehat{\Delta}_n - (n-1)\widehat{\Delta}^{(-i)}_{n-1}, \quad i= 1,\ldots, n. 
\end{equation}
In this context, the value of $\widehat{\Delta}_i$ is derived from (\ref{test statistic}) by excluding the $i$-th observation from the sample $X_1,\ldots, X_n$ and the jackknife estimator $\widehat{\Delta}_{jack}=\dfrac{1}{n}\sum\limits_{i=1}^n\widehat{\mathcal{V}}_i$. 
Let $p_i$ be the probability associated with $\widehat{\mathcal{V}}_i$ such that $\sum\limits_{i=1}^{n}p_i=1$ and $p_i \geq 0 $ for $1\leq i \leq n$. 
Then, the JEL based on $\Delta$ is defined as
\begin{equation}\label{eq.13}
    \mathcal{L}= \max \Bigg\{ \prod_{i=1}^n p_i: \sum_{i=1}^{n}p_i=1, p_i\geq 0,   \sum_{i=1}^{n} p_i\widehat{\mathcal{V}}_i=0  \Bigg \}. 
\end{equation}
We know that $\prod\limits_{i=1}^np_i$, constrained by $\sum\limits_{i=1}^{n}p_i=1$, reaches its maximum value of ${1}/{n^n}$ when each $p_i= \dfrac{1}{n}$.
Thus, the JEL ratio for $\Delta$ is given by
\begin{equation}\label{R}
    \mathcal{R}= \dfrac{\mathcal{L}}{n^{-n}}= \max \bigg\{ \prod_{i=1}^n np_i: \sum_{i=1}^n p_i=1, p_i\geq 0, \sum_{i=1}^{n} p_i\widehat{\mathcal{V}}_i=0 \bigg\}.
\end{equation}
Using Lagrange multipliers, whenever
  \begin{equation}\label{LM}
      \min_{1\leq k \leq n} \widehat{\mathcal{V}}_i^k < \widehat{\Delta}_{jack}< \max_{1\leq k \leq n} \widehat{\mathcal{V}}_i^k, 
  \end{equation}
  we have $p_i= \frac{1}{n}\frac{1}{1+\lambda_1 \widehat{\mathcal{V}}_i}$,  and $\lambda_1$ satisfies  $\frac{1}{n} \sum\limits_{i=1}^n \frac{\widehat{\mathcal{V}}_i}{1+\lambda_1\widehat{\mathcal{V}}_i}=0$. 

 Substituting the value of $p_i$ in $(\ref{R})$ and taking the logarithm of $\mathcal{R}$, we get the non-parametric jackknife empirical log-likelihood ratio as
 \begin{equation*}
     \log \mathcal{R}= -\sum_{i=1}^n \text{log}\{1+\lambda_1\widehat{\mathcal{V}}_i\}.
 \end{equation*}
 
For large values of $\log\mathcal{R}$, we reject the null hypothesis $H_0$ in favour of the alternative hypothesis $H_1$. To determine the critical region for the JEL-based test, we ascertain the limiting distribution of the jackknife empirical log-likelihood ratio.
 The following regularity conditions, as outlined by \cite{gastwirth1974large}, are required to explain how Wilk’s theorem works:   \\
($A1$) The random variable $X$ possesses a finite mean $\mu$ and variance $\sigma^2$, and \\
($A2$) The probability density function of $X$ is continuous in the neighbourhood of $\mu$.

Using Theorem $1$ from \cite{jing2009jackknife}, we have established the asymptotic null distribution of the jackknife empirical log-likelihood ratio as an analogue to Wilks' theorem. 

\begin{theorem}\label{thm3}
Assume that $A1$ and $A2$ hold. Then as $n\to \infty$, $-2 \log \mathcal{R}$ converges to $\chi^2_1$ in distribution.
\end{theorem}
\begin{proof}
  The proof of this theorem can be established by following the argument of \cite{zhao2015jackknife}. First we have $\frac{1}{n}\sum_{i=1}^n\widehat{\mathcal{V}}_i^2=\widehat{\sigma}^2+o_p(n^{-1/2})$ almost surely. Thus
  \begin{align*}
     \lambda_1&=\Bigg\{\frac{1}{n}\sum_{i=1}^n\widehat{\mathcal{V}}_i \Big/\frac{1}{n}\sum_{i=1}^n\widehat{\mathcal{V}}_i^2 \Bigg\}+o_p(n^{-1/2})\\
     &=\dfrac{\widehat{\Delta}_{jack}}{\widehat{\sigma}^2} + o_p(n^{-1/2}).
  \end{align*}  
   Using Lemma $A1$ of Jing \emph{et al}. \cite{jing2009jackknife}, we have the condition (\ref{LM}) and $\mathbb{E}(Y^2)<\infty$  and $\sigma^2>0$ are also satisfied. 
   The Wilks' theorem of JEL can be established as
   \begin{align*}
       -2\log \mathcal{R} &= 2\sum_{i=1}^n \Bigg\{\lambda_1 \widehat{\mathcal{V}}_i-\frac{1}{2}\lambda_1^2 \widehat{\mathcal{V}}_i^2 \Bigg\}+ o_p(1)\\
       &= 2n{\lambda_1} \widehat{\Delta}_{jack}- n{\lambda_1}^2\widehat{\sigma}^2+o_p(1)\\
       &=\dfrac{ n\big\{ \widehat{\Delta}_{jack}\big\}^2 }{\widehat{\sigma}^{2}}+o_p(1)\\
       &\xrightarrow{d} \chi^2_1.
   \end{align*}
   Hence, the theorem.
\end{proof}
Hence we reject the null hypothesis $H_0$ in favour of the alternative hypothesis $H_1$ at a significance level $\alpha$, if
\begin{equation*}
    -2 \log \mathcal{R} > \chi_{1,\alpha}^2
\end{equation*}
where, $\chi_{1,\alpha}^2$ represents the upper $\alpha$-percentile point of the $\chi^2$ distribution with one degree of freedom.

\subsection{Adjusted jackknife empirical likelihood ratio test}
\noindent To overcome the convex hull restrictions inherent in the traditional JEL approach, \cite{chen2008adjusted} introduced the AJEL ratio test. Generally, the AJEL method improves upon the original JEL in terms of empirical power and coverage probability.
From the generated $\mathcal{V}_i$'s in (\ref{vi}), the extra data point is obtained by using the proposed convention; $\widehat{\mathcal{W}}_{n+1}= -k_n  \widehat{\Delta}_{jack}$, 
where $k_n$ is a positive constant. Then the AJEL is defined as,
\begin{equation*}
    \mathcal{L}^{ad}= \max \Bigg\{ \prod_{i=1}^{n+1} p_i: \sum_{i=1}^{n+1}p_i=1, p_i\geq 0,   \sum_{i=1}^{n+1} p_i\widehat{\mathcal{W}}_i=0  \Bigg \}, 
\end{equation*}
where 
 \begin{equation}\label{wi}
    \widehat{\mathcal{W}}_i= \begin{cases}
\widehat{\mathcal{V}}_i  \hspace{5.4cm}  \text{for}  ~ i=1,2,\ldots,n \\
-\max\Big(1,\dfrac{\log(n)}{2}\Big) \dfrac{1}{n} \sum\limits_{i=1}^{n} \widehat{\mathcal{V}}_i \hspace{1cm} ~~~ \text{for} ~ i=n+1.
\end{cases}
\end{equation} 
Therefore, we define the AJEL ratio as  
\begin{equation*}
     \mathcal{R}^{ad}=\dfrac{\mathcal{L}^{ad}}{(n+1)^{-(n+1)}}= \max \bigg\{ \prod_{i=1}^{n+1} (n+1)p_i: \sum_{i=1}^{n+1} p_i=1, p_i\geq 0, \sum_{i=1}^{n+1} p_i\widehat{\mathcal{W}}_i=0 \bigg\}.
\end{equation*}
Then, we obtain the AJEL log ratio,
\begin{equation}
    -2\log \mathcal{R}^{ad}=2\sum_{i=1}^{n+1}\log\{1+\lambda_2 \widehat{\mathcal{W}}_i\},
\end{equation}
where $\lambda_2$ satisfies the equation,
\begin{equation*}
    f(\lambda)= \dfrac{1}{n+1}\sum_{i=1}^{n+1} \dfrac{\widehat{\mathcal{W}}_i}{1+\lambda_2 \widehat{\mathcal{W}}_i}=0.
\end{equation*}
We can derive the following Wilks' theorem for the AJEL, which provides a foundational result in establishing the asymptotic behavior of the test statistic. 
\begin{theorem}\label{thm4}
    Under regularity conditions hold like Theorem \ref{thm3}. Under $H_0$, the $-2\log \mathcal{R}^{ad}$ converges in distribution to $\chi^2_1$ as $n\to \infty$. 
\end{theorem}
\begin{proof}
 By following the argument of \cite{chen2008adjusted} and \cite{zhao2015jackknife}, we have the adjusted jackknife pseudo value defined in (\ref{wi}). 
We have $|{\lambda}_2|={O}_p(1/\sqrt{n})$, when $k_n=o_p(n)$. The Wilks' theorem of AJEL can be derived as 
\begin{align*}
    -2 \log\mathcal{R}^{ad}&= \sum_{i=1}^{n+1} 2\log (1 +{\lambda}_2 \widehat{\mathcal{W}}_i)\\
    &= 2\sum_{i=1}^{n+1} \Big( {\lambda}_2 \widehat{\mathcal{W}}_i - \dfrac{({\lambda}_2 \widehat{\mathcal{W}}_i^2}{2} \Big) +o_p(1)\\
    &= 2n{\lambda}_2 \widehat{\Delta}_{jack}- n{\lambda}^2_2\widehat{\sigma}^2+o_p(1)\\
    &=\dfrac{ n\big\{ \widehat{\Delta}_{jack}\big\}^2 }{\widehat{\sigma}^{2}}+o_p(1). 
\end{align*}
 Hence by Slutsky's theorem, as $n\to \infty$,
\begin{equation*}
    -2 \log\mathcal{R}^{ad}=\dfrac{ n\big\{ \widehat{\Delta}_{jack}\big\}^2 }{\widehat{\sigma}^{2}} 
    \xrightarrow{d} \chi^2_1. 
\end{equation*} 
Hence, the theorem.
\end{proof}
From Theorem \ref{thm4},  we can then reject the null hypothesis $H_0$ at a $\alpha$ level of significance if
\begin{equation*}
    -2\log \mathcal{R}^{ad}> \chi^2_{1,\alpha}.
\end{equation*}

\section{Simulation study}\label{sec4}
\noindent In this section, we conducted an extensive Monte Carlo simulation study to examine the finite sample performance of the newly proposed tests against fixed alternatives; the computations were exclusively conducted utilizing the \texttt{R} programming language. In our simulation study, we employed the \texttt{R} package \( \texttt{symmetry} \) (a version of which is accessible on CRAN, see \cite{ivanovic2020symmetry}). The main objectives were to assess the empirical type-I error and empirical power of the proposed tests and compare their performance with existing tests for symmetry. Sample sizes of $25, 50, 100$ and $200$ were considered to evaluate how the sample size affects the performance of the tests. The simulation is repeated $10000$ times. To calculate the empirical type-I error of the tests, we considered various symmetric distributions such as standard normal $(N(0,1))$, standard Laplace $(Lap(0,1))$, standard lognormal $(Log(0,1))$, and a mixture of normal $(MxN(0.5))$. The critical values of the SCR approach are calculated using the generated samples from the considered symmetric distributions. The empirical type-I error calculated for both SCR, JEL and AJEL ratio tests is reported in Table \ref{table:size}. From Table \ref{table:size}, we observe that as the sample size increases, the empirical type-I error converges to the given significance level.  

\begin{table}[ht]
\centering
\caption{Empirical type-I rate at $0.05$ level of significance}
\label{table:size} 
    \resizebox{15.5 cm}{!}{
\fontsize{9pt}{11pt}\selectfont
\begin{tabular}{p{1.1cm} p{0.5cm} p{0.7cm}p{0.7cm} p{0.7cm} p{0.7cm} p{0.7cm} p{0.7cm}p{0.7cm} p{0.7cm} p{0.7cm} p{0.7cm} p{0.7cm} p{0.7cm}} 
\hline
        &$n$ &SCR &JEL&AJEL & MW & MS & BHI   & CM  &  MGG &   MOI  &  NAI  &B1&  SGN  \\[1ex]
\hline
N(0,1) & 25   & 0.047& 0.065& 0.067& 0.021& 0.040& 0.057 & 0.042 & 0.045 & 0.041  & 0.053 & 0.031& 0.043 \\[1ex]
        & 50  & 0.048& 0.062 &0.051 &0.038 &0.038 & 0.050 & 0.041 & 0.043 & 0.049  & 0.051 & 0.040 & 0.048 \\[1ex]
        & 100  & 0.050 & 0.053&0.054 & 0.044&0.053 & 0.051 & 0.045 & 0.047 & 0.051 & 0.049  &0.055 &0.052\\[1ex]
        & 200  & 0.052 & 0.051 & 0.049&0.049 & 0.051& 0.050 & 0.047 & 0.050 & 0.049  & 0.051 & 0.054&0.053 \\[3ex]

Lap(0,1) &  25   & 0.041& 0.056& 0.054&0.006 &0.051 & 0.057 & 0.032 & 0.053 &  0.039 & 0.054  & 0.107 &0.043\\[1ex]
        & 50   & 0.046& 0.061&0.058 &0.038 &0.058 & 0.053 & 0.038 & 0.056 &  0.047 & 0.053 & 0.087 & 0.058\\[1ex]
        & 100  & 0.050& 0.054 & 0.052&0.060 &0.067 & 0.052 & 0.035 & 0.058 &  0.045 & 0.050  & 0.079&0.057 \\[1ex]
        & 200  & 0.051& 0.052& 0.049& 0.048&0.053 & 0.050 & 0.041 & 0.051 &  0.051 &  0.049 &0.063 & 0.052\\[3ex]
        
Log(0,1) & 25   & 0.037& 0.063 &0.062& 0.004& 0.029& 0.057 & 0.036 & 0.046 & 0.031  & 0.038 &  0.053 &0.050\\[1ex]
        & 50   & 0.042& 0.066 &0.052& 0.020&0.025 & 0.054 & 0.038 & 0.044 &  0.049 & 0.049 & 0.047 & 0.052\\[1ex]
        & 100 & 0.051& 0.054  &0.048&0.038 &0.048 & 0.055 & 0.040 & 0.048 &  0.048 & 0.051 & 0.079& 0.049\\[1ex]
        & 200  & 0.052& 0.050 &0.051& 0.048& 0.050& 0.050 & 0.041 & 0.049 &  0.049 & 0.050  &0.063 & 0.050\\[3ex]
        
MxN(0.5) & 25   & 0.065& 0.062 &0.060& 0.004& 0.035& 0.060 & 0.042 & 0.048 & 0.042 & 0.055 & 0.027  &0.047\\[1ex]
      & 50   & 0.049& 0.056 &0.053&0.010 &0.039 & 0.055 & 0.041 & 0.044 &  0.049 & 0.056 &0.043  & 0.060\\[1ex]
        & 100  & 0.050& 0.054 &0.050&0.025 & 0.044& 0.048 & 0.043 & 0.042 &  0.050 & 0.054 &0.049 & 0.055\\[1ex]
        & 200  & 0.052& 0.051 &0.049& 0.036&0.049 & 0.050 & 0.043 & 0.050 &  0.049 & 0.052 & 0.046 & 0.053\\[1ex]
\hline
\end{tabular}}
\end{table}
 
Next, to calculate the empirical power of the proposed test, we used the different forms of skewed distributions with $\theta= 0.5 ~\text{and}~ 1$ as:
 \begin{itemize}
      \item Fernandez-Steel type distributions (see \cite{fernandez1998bayesian}) with density function
    \begin{eqnarray*}
        g(x; \theta) &=& f\Big(\dfrac{x}{1+\theta}\Big)I(x<0)+f\big((1+\theta)x\big)I(x\geq 0);
    \end{eqnarray*}
    
    \item Azzalini type distributions (see \cite{azzalini2013skew}) with density function
    \begin{eqnarray*}
        g(x; \theta)&=& 2f(x)F(\theta x);
    \end{eqnarray*}
\item Contamination (with shift) alternative with distribution function
    \begin{eqnarray*}
        G(x; \theta,\beta)&=&(1-\theta)F(x)+\theta F(x-\beta), ~~ \beta>1,~ \theta \in [0,1];
    \end{eqnarray*}
 \end{itemize}
 where $f$ and $g$ are density functions, $G$ and $F$ are the distribution functions. For the empirical power calculation, the critical values for the SCR test are calculated using the standard normal distribution. As these values are derived under the null hypothesis, the particular choice of symmetric distribution does not significantly affect the outcomes.

Several tests are available in the literature to test the symmetry of random variables based on various characterizations, and well-known classical tests exist. The subsequent tests are considered for comparative study:
\begin{itemize}
    \item MW: Modified signed rank Wilcoxon test (\cite{vexler2023implement}),
    \item MS: Modified sign test (\cite{vexler2023implement}),
    \item BHI: Statistic of the Litvinova test (\cite{litvinova2001new}),
    \item CM: Cabilio-Masaro test statistic (\cite{cabilio1996simple}),
    \item MGG: Miao, Gel and Gastwirth test statistic 
 (\cite{miao2006new}),
    \item MOI:  Milosevic and Obradovic test statistic (\cite{milovsevic2016characterization}),
    \item NAI: Statistic of the Nikitin and Ahsanullah test (\cite{nikitin2015new}),
    \item B1: Statistic of the test $\sqrt{b_1}$ (\cite{milovsevic2019comparison}),
    \item SGN: Sign test statistic (\cite{milovsevic2019comparison}).
\end{itemize}

\begin{landscape}
\begin{table}[ht]
\centering
\caption{Empirical powers of the tests under Fernandez-Steel type distribution }
\label{table:power1} 
     \resizebox{17.5 cm}{!}{
\fontsize{9pt}{10pt}\selectfont
\begin{tabular}{p{1.5cm} p{0.8cm} p{0.8cm} p{0.8cm} p{0.8cm} p{0.8cm} p{0.8cm} p{0.8cm}p{0.8cm} p{0.8cm} p{0.8cm} p{0.8cm} p{0.8cm} p{0.8cm}} 
\hline
        &$n$ & SCR &JEL & AJEL & MW & MS & BHI   & CM  &  MGG &   MOI  &  NAI  & B1 & SGN  \\[1.5ex]
\hline
N(0.5) & 25  & 0.581 & 0.488& 0.436& 0.035 & 0.257 & 0.322 & 0.338 & 0.390 & 0.181 & 0.313 & 0.298 & 0.431 \\[1ex]
       & 50   & 0.842& 0.797 &0.779& 0.251 & 0.463 & 0.056 & 0.681 & 0.702 & 0.313 & 0.557 & 0.689 & 0.644 \\[1ex]
       & 100  & 0.998& 0.984 &0.979& 0.768 & 0.790 & 0.851 & 0.923 & 0.927 & 0.554 & 0.859 & 0.968 & 0.780 \\[1ex]
       & 200  & 1.000& 1.000 &0.999& 0.992 & 0.997 & 0.990 & 0.995 & 0.994 & 0.841 & 0.990 & 0.999 & 0.868\\[3ex]

Lap(0.5) & 25   & 0.632& 0.590 &0.423& 0.059 & 0.303 & 0.217 & 0.270 & 0.422 & 0.131 & 0.218 & 0.377 & 0.436 \\[1ex]
         & 50   &0.798& 0.785 &0.656 & 0.331 & 0.626 & 0.393 & 0.605 & 0.724 & 0.222 & 0.390 & 0.668 & 0.612 \\[1ex]
         & 100 & 0.990& 0.985 &0.897& 0.798 & 0.924 & 0.667 & 0.931 & 0.962 & 0.416 & 0.677 & 0.843 & 0.841 \\[1ex]
         & 200 & 1.000 & 1.000 &0.993& 0.992 & 0.998 & 0.927 & 0.999 & 0.999 & 0.694 & 0.926 & 0.958 & 0.899\\[3ex]
        
Log(0.5) & 25   & 0.587& 0.464 &0.431& 0.043 & 0.239 & 0.280 & 0.307 & 0.388 & 0.159 & 0.282 & 0.351 & 0.428 \\[1ex]
         & 50   & 0.782& 0.720 &0.731& 0.269 & 0.492 & 0.506 & 0.660 & 0.703 & 0.290 & 0.506 & 0.650 & 0.614 \\[1ex]
         & 100  & 0.980& 0.945 &0.946& 0.779 & 0.823 & 0.803 & 0.930 & 0.938 & 0.514 & 0.808 & 0.902 & 0.792 \\[1ex]
         & 200 & 1.000 & 1.000 &0.999& 0.992 & 0.986 & 0.978 & 0.996 & 0.998 & 0.809 & 0.979 & 0.994 & 0.849\\[3ex]
       
N(1)   & 25   & 0.738& 0.820 &0.781& 0.047 & 0.409 & 0.598 & 0.689 & 0.731 & 0.369 & 0.598 & 0.620 & 0.406\\[1ex]
       & 50   & 0.887& 0.985 &0.983& 0.270 & 0.699 & 0.889 & 0.927 & 0.926 & 0.615 & 0.882 & 0.946 & 0.662 \\[1ex]
       & 100  & 0.994& 1.000 &0.999& 0.721 & 0.955 & 0.994 & 0.994 & 0.993 & 0.882 & 0.993 & 1.000 & 0.781 \\[1ex]
       & 200 & 1.000 & 1.000 &1.000& 0.982 & 0.999 & 1.000 & 1.000 & 1.000 & 0.993 & 1.000 & 1.000 & 0.868\\[3ex]
       
Lap(1) & 25   & 0.788& 0.844 &0.779& 0.071 & 0.524 & 0.467 & 0.639 & 0.780 & 0.280 & 0.466 & 0.701 & 0.462\\[1ex]
       & 50   &0.846& 0.975  &0.969& 0.369 & 0.895 & 0.762 & 0.956 & 0.974 & 0.496 & 0.766 & 0.955 & 0.688 \\[1ex]
       & 100  & 0.990& 1.000 &0.999& 0.872 & 0.997 & 0.969 & 0.999 & 0.999 & 0.784 & 0.968 & 0.993 & 0.710 \\[1ex]
       & 200 & 1.000 & 1.000 &1.000& 0.999 & 1.000 & 0.999 & 1.000 & 1.000 & 0.971 & 0.999 & 1.000 & 0.738\\[3ex]
        
Log(1) & 25   & 0.749& 0.821 &0.778& 0.051 & 0.408 & 0.556 & 0.669 & 0.739 & 0.337 & 0.558 & 0.665 & 0.544 \\[1ex]
       & 50   & 0.898& 0.982 &0.980& 0.295 & 0.741 & 0.853 & 0.936 & 0.939 & 0.590 & 0.856 & 0.874 & 0.569 \\[1ex]
       & 100 & 1.000 & 1.000 &0.999& 0.787 & 0.974 & 0.989 & 0.995 & 0.995 & 0.858 & 0.990 & 0.991 & 0.654 \\[1ex]
       & 200 & 1.000 & 1.000 &1.000& 0.994 & 1.000 & 1.000 & 1.000 & 1.000 & 0.991 & 0.999 & 1.000 & 0.757 \\[3ex]
\hline
\end{tabular} 
}
\end{table}
\end{landscape}

\begin{landscape}
\begin{table}[ht]
\centering
\caption{Empirical powers of the tests under Azzalini type distribution }
\label{table:power2}  
    \resizebox{17.5 cm}{!}{
\fontsize{9pt}{10pt}\selectfont
\begin{tabular}{p{1.5cm} p{0.6cm} p{0.8cm}p{0.8cm} p{0.8cm} p{0.8cm} p{0.8cm} p{0.8cm}p{0.8cm} p{0.8cm} p{0.8cm} p{0.8cm} p{0.8cm} p{0.8cm}} 
\hline
        &$n$&SCR &JEL  & AJEL& MW & MS & BHI   & CM  &  MGG &   MOI  &  NAI  & B1 & SGN  \\[1.5ex]
\hline
N(0.5)  & 25   &0.784& 0.545 &0.483& 0.014 &0.296 &  0.381  & 0.452& 0.354 & 0.440  & 0.381  & 0.311& 0.541 \\[1ex]
        & 50   &0.898& 0.819 &0.802& 0.089 & 0.441&  0.577  & 0.682 & 0.391 & 0.565  & 0.592  & 0.531& 0.621 \\[1ex]
        & 100  &1.000& 0.979 &0.980& 0.350 & 0.681& 0.843   & 0.895 & 0.451 & 0.932  &  0.784 &0.803 & 0.874 \\[1ex]
        & 200 & 1.000 & 1.000&0.999& 0.654 &0.897 &  1.000  & 1.000& 0.637 &  0.990 &  0.942 & 0.984& 1.000\\[3ex]

Lap(0.5) & 25   &0.571& 0.418 &0.379&0.012 & 0.038&  0.510  &0.147 & 0.528 & 0.328  & 0.285  & 0.298& 0.175 \\[1ex]
         & 50   &0.668& 0.603 &0.574&0.078 &0.191 &  0.632  &0.151& 0.687 &  0.490 &  0.538 &0.548 & 0.226 \\[1ex]
        & 100  & 0.891 & 0.979&0.881& 0.156& 0.276& 0.871   &0.169& 0.873 & 622  & 0.628  & 0.827& 0.362 \\[1ex]
        & 200 & 1.000 &1.000 &0.999& 0.206&0.359 &  1.000  &0.192& 1.000 &  0.879 & 0.820  & 0.985& 0.490 \\[3ex]
        
Log(0.5) & 25  &0.482& 0.406  &0.373&0.240 &0.324 &  0.305  &0.385& 0.373 & 0.421  &  0.398 & 0.272& 0.221 \\[1ex]
        & 50   &0.692& 0.713 &0.654&0.416 &0.492 & 0.638   &0.585& 0.551 &  0.783 &  0.734 &0.545 & 0.401 \\[1ex]
        & 100  & 0.801 & 0.910&0.917&0.610 &0.660 &  0.873  &0.814& 0.783 & 0.869  & 0.894  &0.867 & 0.570 \\[1ex]
        & 200 & 1.000 & 1.000 &0.998& 0.791&0.897 & 0.990   &0.965& 0.956 &  0.991 & 0.988  & 0.994& 0.682 \\[3ex]

N(1)   & 25  &0.428 & 0.356 &0.334& 0.004 & 0.418& 0.481 &0.308& 0.373 & 0.798  & 0.742  &0.505 & 0.304 \\[1ex]
       & 50   &0.783& 0.569 &0.559 &0.022 & 0.497& 0.610 &0.414& 0.554 &  0.849 &  0.877 &0.774 & 0.549 \\[1ex]
        & 100  &0.909& 0.791 &0.847& 0.062 & 0.612& 0.804 &0.594& 0.788 & 0.914  & 0.942  & 0.971&  0.671\\[1ex]
        & 200 & 1.000 &1.000 &0.968 &0.141 & 0.773& 1.000 &0.800& 1.000 & 1.000  & 1.000  &0.99 & 0.715 \\[3ex]

Lap(1) & 25   & 0.792 & 0.633 &0.594&0.271 &0.470 &  0.512  &0.205& 0.373 &  0.781 & 0.801  & 0.495& 0.441 \\[1ex]
      & 50   & 0.847 & 0.854 &0.780&0.369 &0.681 &  0.708  &0.203& 0.464 &  0.846 &   0.891&0.771 & 0.613 \\[1ex]
        & 100  &0.978& 0.990 &0.891&0.572 &0.798 &  0.871  &0.221& 0.636 & 0.961  & 0.948  &0.971 & 0.749 \\[1ex]
        & 200 & 1.000 & 1.000 &0.989&0.714 &1.000 &  0.968  &0.236& 0.832 & 1.000  & 1.000 & 1.000 & 0.929 \\[3ex]
        
Log(1) & 25  &0.656  & 0.766 &0.671&0.291 &0.384 & 0.309   &0.278& 0.253 & 0.761  & 0.812  & 0.124& 0.412 \\[1ex]
        & 50   &0.835& 0.852 &0.716&0.390 & 0.469&  0.422  &0.359& 0.316 & 0.810  & 0.890  & 0.196& 0.481 \\[1ex]
        & 100  &1.000 & 0.917 &0.809&0.455 & 0.562& 0.510   &0.497& 0.446 & 0.931  & 0.948  & 0.397& 0.531 \\[1ex]
        & 200 & 1.000 & 1.000 &0.971& 0.608& 0.639&  0.794  &0.681& 0.630 & 1.000  & 1.000  & 0.634& 0.611 \\[1ex]
\hline
\end{tabular}}
\end{table}
\end{landscape}

\begin{landscape}
\begin{table}[ht]
\centering
\caption{Empirical powers of the tests under contamination alternative type distribution }
\label{table:power3} 
    \resizebox{17.5 cm}{!}{
\fontsize{9pt}{10pt}\selectfont
\begin{tabular}{p{1.5cm} p{0.6cm} p{0.8cm}p{0.8cm} p{0.8cm} p{0.8cm} p{0.8cm} p{0.8cm}p{0.8cm} p{0.8cm} p{0.8cm} p{0.8cm} p{0.8cm} p{0.8cm}} 
\hline
        &$n$&SCR &JEL &AJEL & MW & MS & BHI   & CM  &  MGG &   MOI  &  NAI  & B1 & SGN  \\[1.5ex]
\hline
N(0.5)    & 25  & 0.586  & 0.411 &0.346&0.335 &0.496 &  0.567  &0.481& 0.412 & 0.392  & 0.375  &0.531 & 0.481 \\[1ex]
        & 50   &0.865& 0.621 &0.592&0.523 &0.581 & 0.642 & 0.597& 0.598 & 0.493  & 0.578  &0.701 & 0.646 \\[1ex]
        & 100  & 1.000& 0.908 &0.898 & 0.780&0.761 &  0.873  &0.694  & 0.740 & 0.680  &  0.841 & 0.907& 0.941 \\[1ex]
        & 200 & 1.000 & 1.000&0.996&0.991 &0.959 &  1.000  & 0.898 & 0.908 & 0.867  & 0.923  & 0.998& 1.000 \\[3ex]

Lap(0.5) & 25  & 0.662 & 0.517 &0.461&0.431 &0.514 & 0.407 &0.328 & 0.527 & 0.481& 0.557  & 0.597& 0.420 \\[1ex]
         & 50   &0.841 & 0.810 &0.775&0.684 & 0.669& 0.597 & 0.490& 0.619 & 0.671& 0.709  & 0.724& 0.580 \\[1ex]
        & 100  & 0.999 & 0.981 &0.977&0.806 &0.816 & 0.829   &0.789 & 0.791 &0.806 & 0.890  &0.894 &0.891  \\[1ex]
        & 200 & 1.000 & 1.000&1.000& 0.898 &0.994 &  0.968  & 0.958& 0.893 & 0.934  & 0.968  & 0.991& 0.977 \\[3ex]
        
Log(0.5) & 25   & 0.592 & 0.446 &0.357&0.384 & 0.481& 0.517&0.429 & 0.488 & 0.397  & 0.418  &0.518 & 0.499 \\[1ex]
        & 50   &0.941& 0.835 &0.703&0.509 &0.633 & 0.760 &0.629 & 0.607 & 0.559  & 0.634  &0.626 & 0.597 \\[1ex]
        & 100 &1.000  & 0.986&0.963&0.697 &0.894 & 0.919 &0.816 & 0.790 & 0.873  & 0.872  & 0.903& 0.833 \\[1ex]
        & 200 & 1.000 &1.000 &0.999&0.994 &0.989 &  1.000  & 1.000 & 0.973 & 0.955  & 1.000  & 0.996& 0.990 \\[3ex]

N(1)   & 25   &0.610& 0.919 &0.896&0.519 &0.594 & 0.684 & 0.604& 0.580 & 0.448  & 0.525  & 0.696& 0.570 \\[1ex]
       & 50   &0.941 & 0.998&0.989&0.809 &0.784 & 0.843 &0.790 & 0.729 &  0.697 & 0.785  &0.738 & 0.690 \\[1ex]
        & 100  &1.000& 1.000 &1.000&0.991 &0.968 & 0.992 &0.938 & 0.988 & 0.893  & 0.832  & 0.956& 0.968 \\[1ex]
        & 200 & 1.000 &  1.000 &1.000& 1.000 & 1.000 & 1.000   & 1.000 & 1.000 & 1.000  & 1.000  & 1.000 & 1.000 \\[3ex]

Lap(1) & 25   & 0.718 & 0.861 &0.831&0.492 & 0.643& 0.589 & 0.493& 0.679 & 0.609  & 0.678  &0.621 & 0.588 \\[1ex]
      & 50   &0.891& 0.981 & 0.946&0.746&0.797 & 0.894 & 0.609& 0.891 & 0.867  & 0.894  &0.840 & 0.798 \\[1ex]
        & 100 &0.984& 1.000  &1.000&0.869 & 0.918& 0.993 &0.972 & 0.994 & 0.969  & 0.984  & 0.992& 0.980 \\[1ex]
        & 200 & 1.000 & 1.000 &1.000& 0.948& 1.000 & 1.000 & 1.000 &0.987 &  1.000 &  1.000 & 1.000& 1.000 \\[3ex]
        
Log(1) & 25   & 0.748& 0.861 &0.833 &0.571& 0.509& 0.664 & 0.571& 0.623 & 0.527  & 0.671  &0.631 & 0.609 \\[1ex]
        & 50   &1.000& 0.893 &0.864 &0.791& 0.871& 0.890 &0.790 & 0.809 & 0.738  & 0.809  &0.907 & 0.799 \\[1ex]
        & 100 & 1.000 &0.999 &0.978&0.949 & 1.000& 0.992 &0.937 & 0.964 & 0.966  & 0.985  & 1.000 & 0.938 \\[1ex]
        & 200 & 1.000 & 1.000 &1.000& 1.000 & 1.000 & 1.000   &1.000& 1.000 &1.000   & 1.000  &1.000 & 1.000 \\[2ex]
\hline
\end{tabular}}
\end{table}
\end{landscape}

The results of empirical power compression are given in Tables \ref{table:power1}-\ref{table:power3}. From Tables \ref{table:power1}-\ref{table:power3}, we observe that our entropy-based tests showed significantly higher power than classical tests when detecting asymmetry, especially in moderate to severe skewness cases. The empirical power of the proposed tests increased with larger sample sizes. Our proposed tests generally perform better than traditional methods across various alternatives. For the Fernandez-Steel and Azzalini type alternative with samples of size $100$, the power reached above $0.85$ for most asymmetric distributions. At the same time, traditional tests, such as the SGN test, often fail to reach $0.70$. Given the overall performance shown in Tables \ref{table:size}–\ref{table:power3}, it is evident that, in most cases, our proposed tests outperform the existing ones.

\section{Data Analysis}\label{sec5}
\noindent In this section, we illustrate the proposed methodology using two real datasets. \\
\textbf{Example 1.}
     The dataset consists of 50 observations detailing the tensile strength (MPa) of fibre ($X$) as reported in \cite{abdul2023}, and this dataset follows a normal distribution with mean $\mu=3076.88$ and $\sigma=344.362$.  The depiction illustrates the fit of a normal distribution to the survival time data shown in Figure $\ref{fig:data1}$. To find the critical values of the SCR test, we have generated a random sample from the standard normal distribution, and the values are $-0.4746$ and $0.4750$. 
     
\begin{figure}[ht] 
    \centering
    \includegraphics[width=13.5cm, height=9cm]{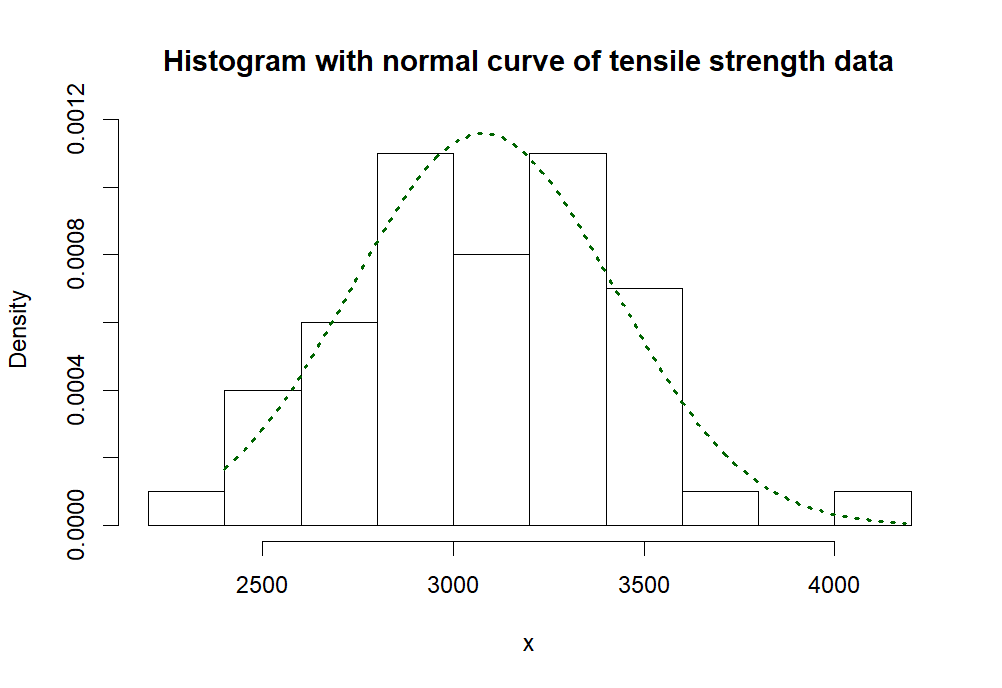}
    \caption{Modeling tensile strength data using normal distribution}
    \label{fig:data1}
\end{figure}

\noindent \textbf{Example 2.}
The dataset is sourced from \cite{qiu2018extropy} and is reported in Table \ref{dataset1}, demonstrating a good fit with an asymmetric inverse Gaussian distribution. The estimated values of parameters $(\widehat{\mu}=3.61, \widehat{\lambda}=1.588 )$ of the inverse Gaussian distribution. This dataset is used to demonstrate the applicability of the proposed test statistic. Figure \ref{fig:exp2} presents the corresponding boxplot. The proposed tests are then applied to evaluate whether the underlying distribution exhibits symmetry. Hence, the exact critical points of the SCR method are $C_1=6.37$ and $C_2=7.99$ for the $5\%$ significance level.

\begin{table}[ht]
\centering
\caption{Active repair times (in hours) for an airborne communication transceiver.} \label{dataset1}
 \resizebox{14 cm}{!}{
\fontsize{9pt}{11pt}\selectfont
\begin{tabular}{p{0.6cm} p{0.6cm} p{0.6cm} p{0.6cm} p{0.6cm}p{0.6cm}  p{0.6cm}p{0.6cm} p{0.6cm} p{0.6cm} p{0.6cm} p{0.6cm} } 
\hline
0.2& 0.3& 0.5& 0.5& 0.5& 0.5& 0.6& 0.6& 0.7& 0.7& 0.7& 0.8 \\[0.5ex]
0.8& 1.0& 1.0& 1.0& 1.0& 1.1& 1.3& 1.5& 1.5& 1.5& 1.5& 2.0 \\[0.5ex]
2.0& 2.2& 2.5& 3.0& 3.0& 3.3& 3.3& 4.0& 4.0& 4.5& 4.7& 5.0\\[0.5ex] 
5.4& 5.4& 7.0& 7.5& 8.8& 9.0& 10.3& 22.0& 24.5  \\
\hline
\end{tabular}}
\end{table}

\begin{figure}[ht]
    \centering
    \includegraphics[width=11cm, height=9cm ]{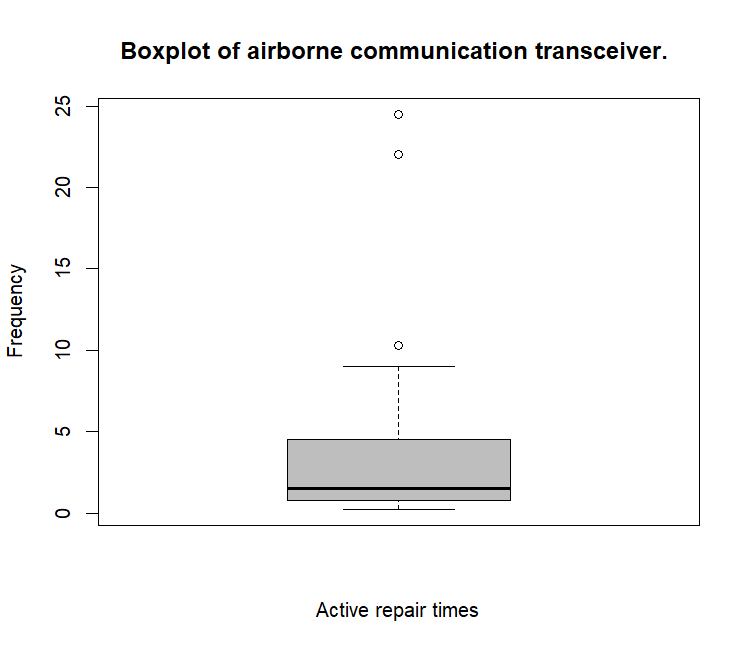}
    \caption{Boxplot}
    \label{fig:exp2}
\end{figure}

\begin{table}[!ht]
 \centering
  \caption{Test statistics along with $p$-values results for real datasets}
    \resizebox{15cm}{!}{
    \fontsize{9pt}{11pt}\selectfont
    \begin{tabular}{rrrrrrrrrrrrrrrrrrrrrrrrrrr}
    \hline
&  \multicolumn{4}{c}{Example 1} &&&& \multicolumn{4}{c}{Example 2} &&  \\
    \cline {2-6}\cline{8-13}  
    & Test statistic && $p$-value &&Decision&& &Test statistic&&$p$-value &&Decision  \\[0.2ex]
\cline{1-13}\\
SCR &0.0307 &&  1.000 ($>$0.05)   &&Accept $H_0$&&& 0.4494  && 0.000 ($<$0.05) && Reject $H_0$  \\[0.2ex]  
JEL &0.1278 && 0.900 ($>$0.05)   &&Accept $H_0$&&& 51.7850 && 0.002 ($<$0.05)  && Reject $H_0$  \\[0.2ex]
AJEL & 0.1176 && 0.916 ($>$0.05)   &&Accept $H_0$&&& 19.3561   && 0.001 ($<$0.05)  && Reject $H_0$  \\[0.2ex]
\hline
    \end{tabular}%
    }
  \label{tab: results}%
\end{table}%

The critical values were obtained by applying the tests to the data, while bootstrap $p$-values were obtained by generating $10000$ samples from the normal distribution using the estimated parameter. Each test was applied to the sample data, and the frequency with which the test statistic exceeded the corresponding critical value was noted. The bootstrap $p$-value was then calculated as the proportion of such occurrences out of the total number of samples. From Table \ref{tab: results}, we observed that the null hypothesis of symmetry is not rejected at the $5\%$ significance level for the tensile strength data, suggesting that a symmetric distribution can appropriately model the tensile strength (in MPa) of the fibre. In contrast, the symmetry hypothesis is rejected for the airborne communication transceiver dataset, indicating that the active repair times (in hours) exhibit an asymmetric distribution.

\section{Conclusion}\label{sec6} 
\noindent In this paper, we have developed two nonparametric tests for assessing the symmetry of a distribution. Using the definitions of generalized cumulative residual entropy and generalized cumulative entropy, we explored and illustrated a novel characterization of the continuous symmetry distribution with examples. We have thoroughly examined the properties of these tests. Given that calculating the null variance of our proposed test can be challenging in practical settings, we have introduced two additional nonparametric tests based on the jackknife empirical likelihood (JEL) and adjusted jackknife empirical likelihood (AJEL) ratio. Through Monte Carlo simulations, we have demonstrated the effectiveness of our proposed tests against various alternative hypotheses. The findings suggest that all the proposed tests demonstrate strong performance across various alternative scenarios. Finally, we have illustrated the applicability of the proposed tests through real-world examples involving tensile strength and airborne communication transceivers datasets.

We conclude this paper by highlighting the potential for future research in exploring new testing problems and developing similar characterization-based tests using alternative measures, such as extropy.
 

\end{document}